\documentclass{amsart}[10pt]

\usepackage{amsrefs}
\usepackage[all]{xy}
\usepackage{syntonly}
\usepackage{hyperref}
\usepackage{amsfonts}
\usepackage{amssymb}
\usepackage{amsmath}
\usepackage{amsthm}
\usepackage[inline]{enumitem}
\usepackage{tikz-cd}
\usepackage{graphics}
\usepackage{graphicx}
\usepackage{color}
\usepackage{comment}
\usepackage{caption,lipsum}
\usepackage{framed}

\usepackage{lineno}

\usepackage{multirow}

\usepackage{footmisc}
\usepackage{todonotes}

\newtheorem{theorem}{Theorem}[section]

\newtheorem{definition}[theorem]{Definition}
\newtheorem{lemma}[theorem]{Lemma}
\newtheorem{proposition}[theorem]{Proposition}

\newtheorem{globalClaim}{Claim}[subsection]

\newtheorem{fact}[theorem]{Fact}

\usepackage{stmaryrd}

\begin{document}

\title[]{Vop\v{e}nka's Principle, Maximum Deconstructibility, and singly-generated torsion classes}

\author{Sean Cox}
\email{scox9@vcu.edu}
\address{
Department of Mathematics and Applied Mathematics \\
Virginia Commonwealth University \\
1015 Floyd Avenue \\
Richmond, Virginia 23284, USA 
}

\date{\today}

\thanks{Partially supported by NSF grant DMS-2154141.}


\begin{abstract}
Deconstructibility is an often-used sufficient condition on a class $\mathcal{C}$ of modules that allows one to carry out homological algebra \emph{relative to $\mathcal{C}$}.  The principle \textbf{Maximum Deconstructibility (MD)} asserts that a certain necessary condition for a class to be deconstructible is also sufficient.  MD implies, for example, that the classes of Gorenstein Projective modules, Ding Projective modules, their relativized variants, and all torsion classes are deconstructible over any ring.  MD was known to follow from Vop\v{e}nka's Principle and imply the existence of an $\omega_1$-strongly compact cardinal.  We prove that MD is equivalent to Vop\v{e}nka's Principle, and to the assertion that each torsion class of abelian groups is generated by a single group within the class (yielding the converse of a theorem of G\"obel and Shelah).
\end{abstract}

\maketitle


\section{Introduction}

For a fixed unital, associative ring $R$, a class $\mathcal{C}$ of (left) $R$-modules is called a \emph{precovering class} if for every module $K$, there is some (not necessarily unique) morphism $\pi_K: C_K \to K$ with $C_K \in \mathcal{C}$ such that every morphism from a member of $\mathcal{C}$ into $K$ factors (not necessarily uniquely) through $\pi$:
\begin{equation*}
\begin{tikzcd}
C_K \arrow[r,"\pi_K"] & K \\
& C' \in \mathcal{C} \arrow[u, "\rho"'] \arrow[ul, dotted]
\end{tikzcd}
\end{equation*}

\noindent The same notion appears more generally in category theory, where a precovering class is known as a \emph{weakly coreflective} class, typically under the additional assumption that $\mathcal{C}$ is retract-closed.

Determining whether a given class is precovering (or weakly coreflective) is a common problem.  For example, whether the class of \emph{Gorenstein Projective} modules is a precovering class (over every ring) is a well-known open problem (\cite{MR2529328}, \cite{MR3621667}, \cite{MR1363858}, \cite{MR1753146}, \cite{MR3690524},  \cite{MR3598789}, \cite{MR3459032}, \cite{MR3760311}, \cite{MR2038564}, \cite{MR3839274}, \cite{MR4115324}, \cite{MR2283103},  \cite{MR2737778}, \cite{MR3473859}, \cite{MR4166732}), though some relative consistency results in set theory are known (\cite{cortes2023module}, \cite{Cox_MaxDecon}).

\emph{Deconstructibility} of a class is an often used sufficient condition for the class to be precovering.  The notion of a deconstructible class grew out of Eklof and Trlifaj's contribution to the proof of the Flat Cover Conjecture (\cite{MR1832549}, \cite{MR1798574}).  A class $\mathcal{C}$ is deconstructible if there is a \textbf{set} $\mathcal{C}_0 \subset \mathcal{C}$ such that $\mathcal{C}$ is exactly the closure of $\mathcal{C}_0$ under \emph{filtrations} (also known as \emph{transfinite extensions}, see Section \ref{sec_Prelims}).  This is very similar to the notion of the cellular closure of a set of morphisms, and the key fact---which goes via Quillen's Small Object Argument---is that deconstructibility of a class implies that the class is precovering (\cite{MR2822215}).

Deconstructible classes are closed under filtrations and  ``eventually almost everywhere closed under quotients" (see Section \ref{sec_Prelims}).  The principle \textbf{Maximum Deconstructibility} asserts that the converse holds.  This principle was isolated in \cite{Cox_MaxDecon} mainly because it implies that for any ring $R$ and every class $\mathcal{X}$ of $R$-modules, the class of $\mathcal{X}$-Gorenstein Projective and $\mathcal{X}$-Ding projective modules are deconstructible, and hence (by \cite{MR2822215}) precovering classes.

It was known that Vop\v{e}nka's Principle implies Maximum Deconstructibility (\cite{Cox_MaxDecon}), and that Maximum Deconstructibility implies the existence of an $\omega_1$-strongly compact cardinal (\cite{MR4835277}).  We prove that Vop\v{e}nka's Principle and Maximum Deconstructibility are equivalent, as part of our Theorem \ref{thm_Main} below.  

Theorem \ref{thm_Main} also provides information about \emph{torsion classes}, whose basic theory was worked out by Dickson~\cite{MR0191935}.  A class of abelian groups is a \textbf{torsion class} if it is closed under direct sums,\footnote{By this we will always mean closure under all  \emph{set-indexed} direct sums.} homomorphic images, and group extensions.  A torsion class is \textbf{singly-generated} if it is the closure of a single group under homomorphic images, extensions, and direct sums.  G\"obel and Shelah proved that Vop\v{e}nka's Principle implies that every torsion class is singly-generated (\cite{MR0780487}, Remark 4.4; see also Ad\'amek-Rosick\'y~\cite{MR1085935}).  Our Theorem \ref{thm_Main} shows this is actually an equivalence.  

\begin{theorem}\label{thm_Main}
The following are equivalent:
\begin{enumerate}
 \item\label{item_VP} Vop\v{e}nka's Principle 
 \item\label{item_MD} Maximum Deconstructibility 
 \item\label{item_MD_weak} Every class of abelian groups that is closed under filtrations and quotients\footnote{By \emph{closure under quotients} we mean that if $A \subset B$ and \textbf{both} $A$ \textbf{and} $B$ are in the class, then is $B/A$ is in the class.  This is weaker than \emph{closure under homomorphic images}, which only assumes $B$ is in the class.} is deconstructible.

 \item\label{item_TorsionClassDecon} Every torsion class of abelian groups is deconstructible.
  
  \item\label{item_TorGenSet} Every torsion class of abelian groups is singly-generated.
\end{enumerate}
\end{theorem}

The equivalence of \eqref{item_VP} with \eqref{item_TorGenSet} is similar to an equivalence proved by Ben Yassine and Trlifaj~\cite{MR4817675}, though the classes they dealt with were not torsion classes.  Both their proof and our proof make crucial use of a functor constructed by Prze\'zdziecki~\cite{MR3187657}.

\section{Preliminaries}\label{sec_Prelims}

An isomorphism-closed class $\mathcal{C}$ of modules is:
\begin{enumerate*}
 \item \textbf{closed under (homomorphic) images} if whenever $C \in \mathcal{C}$, any homomorphic image of $C$ is a member of $\mathcal{C}$.  Equivalently, $C/A \in \mathcal{C}$ whenever $A \subset C$ and $C \in \mathcal{C}$.
 \item \textbf{closed under quotients} if for all modules $C_0$ and $C_1$: if $C_0$ is a submodule of $C_1$ and \emph{both} $C_0$ \emph{and} $C_1$ are members of $\mathcal{C}$, then $C_1/C_0$ is a member of $\mathcal{C}$.
\end{enumerate*}
Clearly, closure under homomorphic images implies closure under quotients.  A very weak form of closure under quotients, called \textbf{eventual almost everywhere closure under quotients}, was isolated in \cite{Cox_MaxDecon} as a crucial property of the class of Gorenstein Projective modules.  We will not need the definition in this paper; it will suffice to know that eventual almost-everywhere closure under quotients follows trivially from closure under quotients.\footnote{But in general they are not equivalent; e.g., the class of free abelian groups is not closed under quotients, but it is eventually almost everywhere closed under quotients.}

Given an ordinal $\alpha$, $V_\alpha$ refers to $\alpha$-th level of the set-theoretic cumulative hierarchy; it is the collection of sets of rank less than $\alpha$.  Every set is a member of $V_\alpha$ for some ordinal $\alpha$, and each $V_\alpha$ is a set.  See Jech~\cite{MR1940513} (Chapter 6) for details.

Though we will not need the exact definition of filtrations or filtration closures in this paper, we briefly explain what they are, since they are so prominent in the field.  Given a class $\mathcal{C}$ of modules, a \textbf{$\boldsymbol{\mathcal{C}}$-filtration} is a sequence $\langle X_\alpha \ : \ \alpha \le \zeta \rangle$ of modules indexed by ordinals up to and including some ordinal $\zeta$, such that
\begin{enumerate*}
 \item $X_0 = 0$
 \item for each $\alpha$, $X_{\alpha+1}/X_\alpha$ is isomorphic to some member of $\mathcal{C}$; and
 \item (continuity) for each limit ordinal $\gamma$, $X_\gamma = \bigcup_{\alpha < \gamma} X_\alpha$.
\end{enumerate*}
\noindent  A module $M$ is \textbf{$\boldsymbol{\mathcal{C}}$-filtered} if it is the union of some $\mathcal{C}$-filtration.  A class $\mathcal{C}$ is \textbf{filtration-closed} (or ``closed under transfinite extensions") if every $\mathcal{C}$-filtered module is a member of $\mathcal{C}$.

We employ several closure operations:
\begin{itemize}
 \item $\textbf{Filt}\boldsymbol{(\mathcal{C})}$ denotes the closure of $\mathcal{C}$ under filtrations.  A class $\mathcal{C}$ is \textbf{deconstructible} if there exists a set $\mathcal{C}_0 \subset \mathcal{C}$ such that $\mathcal{C} = \text{Filt}(\mathcal{C}_0)$.
 \item $\textbf{I}\boldsymbol{(\mathcal{C})}$ is the closure of $\mathcal{C}$ under homomorphic images (which is simply the class of homomorphic images of members of $\mathcal{C}$)
 \item $\textbf{E}\boldsymbol{(\mathcal{C})}$ is the closure of $\mathcal{C}$ under extensions (where a class $\mathcal{A}$ is closed under extensions if whenever $0 \to A_0 \to M \to A_1 \to 0$ is a short exact sequence with each $A_i \in \mathcal{A}$, then $M \in \mathcal{A}$).

 \item $\textbf{D}\boldsymbol{(\mathcal{C})}$ is the closure of $\mathcal{C}$ under all (set-indexed) direct sums. For an ordinal $\alpha$, we let $\boldsymbol{\textbf{D}_\alpha(\mathcal{C})}$ denote the class of all $\alpha$-indexed direct sums of members of $\mathcal{C}$.

  \item $\textbf{T}\boldsymbol{(\mathcal{C})}$ is the torsion closure of $\mathcal{C}$; i.e., $\text{T}(\mathcal{C})$ is the smallest class containing $\mathcal{C}$ that is closed under direct sums, extensions, and homomorphic images.  If $\mathcal{C} = \{ C \}$ is a singleton we write $T(C)$ instead of $T(\{ C \})$.  Recall from the introduction that a torsion class $\mathcal{C}$ is \textbf{singly-generated} if it is the torsion closure of a single group; this is equivalent to asserting that $\mathcal{C}$ is the torsion closure of a \textbf{set} of groups.\footnote{Since, if $\mathcal{C}_0$ is a set of groups, its torsion closure is the same as the torsion closure of the single group $\bigoplus_{C \in \mathcal{C}_0} C$.}
\end{itemize}

Let $\mathcal{C}$ be a class of abelian groups.  $\boldsymbol{{}^{\perp_0} \mathcal{C}}$ denotes the class
\[
\{ A \ : \ \text{Hom}(A,C)=0 \text{ for every } C \in \mathcal{C} \}.
\]
\noindent $\boldsymbol{\mathcal{C}^{\perp_0}}$ is defined dually. If $\mathcal{C} = \{ C \}$ is a singleton, we write ${}^{\perp_0} C$ and $C^{\perp_0}$ instead of ${}^{\perp_0} \{ C \}$ and $\{ C \}^{\perp_0}$, respectively.  Recall that a torsion class is a class closed under set-sized direct sums, homomorphic images, and extensions; in fact, $\mathcal{T}$ is a torsion class if and only if it is of the form ${}^{\perp_0} \mathcal{Y}$ for some class $\mathcal{Y}$, if and only if $\mathcal{T} = {}^{\perp_0} \left( \mathcal{T}^{\perp_0} \right)$ (cf.\ \cite{MR0191935}, Theorem 2.3).

\begin{fact}\label{fact_OperationsPreserve}
If $\mathcal{Y}$ is a class of abelian groups and $\mathcal{C} \subseteq {}^{\perp_0} \mathcal{Y}$, then $\text{I}(\mathcal{C})$, $\text{E}(\mathcal{C})$, $D(\mathcal{C})$, and $\text{Filt}(\mathcal{C})$ are all contained in ${}^{\perp_0} \mathcal{Y}$.
\end{fact}
\begin{proof}
That this is true of $I$, $E$, $D$, and $T$ is routine.  That it also holds of $\text{Filt}$ follows from closure of ${}^{\perp_0} \mathcal{Y}$ under colimits and extensions (cf.\ \cite{MR4835277}, Lemma 6).
\end{proof}

\section{Proof of Theorem \ref{thm_Main}}

The implication \eqref{item_VP} $\implies$ \eqref{item_MD} was proved in \cite{Cox_MaxDecon}.  Statement \eqref{item_MD} trivially implies statement \eqref{item_MD_weak}, since closure under quotients implies eventual almost everywhere closure under quotients.  Statement \eqref{item_MD_weak} implies statement \eqref{item_TorsionClassDecon} because torsion classes of abelian groups are closed under filtrations (\cite{MR4835277}, Lemma 6) and under quotients (since closure under homomorphic images trivially implies closure under quotients).  Statement \eqref{item_TorsionClassDecon} implies statement \eqref{item_TorGenSet} as follows:  suppose $\mathcal{A}$ is a torsion class and is deconstructible; say $\mathcal{A} = \text{Filt}(\mathcal{A}_0)$ for some set $\mathcal{A}_0 \subset \mathcal{A}$.  We claim that
\begin{equation}\label{eq_T_is_clos_of_T0}
\mathcal{A} =T( \mathcal{A}_0) \tag{*}
\end{equation}
Since $\mathcal{A}$ is a torsion class, it is closed under direct sums, images, and extensions,  so the $\supseteq$ direction of \eqref{eq_T_is_clos_of_T0} is trivial.  And torsion classes are filtration-closed by Fact \ref{fact_OperationsPreserve}; so $T (\mathcal{A}_0)$ contains $\text{Filt}(\mathcal{A}_0)= \mathcal{A}$.  So $\mathcal{A}$ is generated by a set, and hence (as remarked in Section \ref{sec_Prelims}) by a single group.

To prove the \eqref{item_TorGenSet} $\implies$ \eqref{item_VP} direction of Theorem \ref{thm_Main}, it will help to have an explicit hierarchy of sets whose union is the torsion closure.

\begin{proposition}\label{prop_IE_set}
If $\mathcal{C}$ is a set, then for any fixed ordinal $\alpha$,  $ I(E(D_\alpha(\mathcal{C})))$ has a set of representatives containing $\mathcal{C}$. 
\end{proposition} 
\begin{proof}
Clearly if $\mathcal{C}$ is a set then $D_\alpha(\mathcal{C})$ has a set of representatives.   And for any class $\mathcal{Y}$, every member of $E(\mathcal{Y})$ can be obtained from $\mathcal{Y}$ in a finite number of extensions; so $E(D_\alpha(\mathcal{C}))$ has a set of representatives. Similarly for $I(-)$ (but even simpler, since $I(-)$ is just the class of images).  So $I(E(D_\alpha(\mathcal{C})))$ has a set of representatives (which can be taken to include the set $\mathcal{C}$).
\end{proof}

\begin{definition}
Let $\mathcal{C}$ be a class of abelian groups.  Recursively define a sequence $\langle T^{\mathcal{C}}_\alpha \ : \ \alpha \in \text{ORD} \rangle$ as follows: 
\begin{itemize}
 \item $T^{\mathcal{C}}_0 = \emptyset$
 \item If $T^{\mathcal{C}}_\alpha$ is defined, let $T^{\mathcal{C}}_{\alpha+1}$ be a set of representatives for 
 \[
 I\left( E \left( D_{\alpha} \left( T^{\mathcal{C}}_{\alpha} \cup \left( V_{\alpha+1} \cap \mathcal{C} \right) \right) \right) \right)
 \]
that contains $T^{\mathcal{C}}_{\alpha} \cup \left( V_{\alpha+1} \cap \mathcal{C} \right)$. (Note that as long as $T^{\mathcal{C}}_\alpha$ is itself a set, then $T^{\mathcal{C}}_{\alpha} \cup \left( V_{\alpha+1} \cap \mathcal{C} \right)$ is a set, so by Proposition \ref{prop_IE_set} the displayed class has a set of representatives.)
 
 \item If $\gamma$ is a limit ordinal, 
 \[
 T^{\mathcal{C}}_\gamma:= \bigcup_{\alpha < \gamma} T^{\mathcal{C}}_\alpha.
 \]
\end{itemize}
Set $T^{\mathcal{C}}_\infty:= \bigcup_{\alpha \in \text{ORD}} T^{\mathcal{C}}_\alpha$. 
\end{definition}

\begin{lemma}\label{lem_HierarchLemma}
For any class $\mathcal{C}$:  
\begin{enumerate}
 \item\label{item_contained} $T^{\mathcal{C}}_\alpha \subseteq T^{\mathcal{C}}_\beta$ whenever $\alpha \le \beta$
 \item\label{item_ContainsTValpha} $V_\alpha \cap \mathcal{C} \subseteq T^{\mathcal{C}}_\alpha$ for all $\alpha$

 \item\label{eq_IsTorsionClosure} $ T^{\mathcal{C}}_\infty=T(\mathcal{C}) $ (i.e., $T^{\mathcal{C}}_\infty$ is the torsion closure of $\mathcal{C}$).

\end{enumerate}
\end{lemma}
\begin{proof}
\eqref{item_contained} and \eqref{item_ContainsTValpha} are clear from the construction; in particular $\mathcal{C} \subseteq T^{\mathcal{C}}_\infty$.  It is also clear that $T^{\mathcal{C}}_\infty$ is closed under direct sums, images, and extensions.  Furthermore, to see it is minimal with this property, induction on $\alpha$ together with Fact \ref{fact_OperationsPreserve} imply that if $\mathcal{V}$ is a torsion class containing $\mathcal{C}$, then $\mathcal{V} \supseteq T^{\mathcal{C}}_\alpha$ for all $\alpha \in \text{ORD}$.

\end{proof}

To finish the proof, we will use of the following beautiful theorem of Prze\'zdziecki:

\begin{theorem}[Prze\'zdziecki~\cite{MR3187657}]\label{thm_PrzFunctor}
There is a functor $\mathcal{F}: \textbf{Graphs} \to \textbf{Ab}$ such that for any graphs $U$ and $V$,
\[
\text{hom}_{\textbf{Ab}}\left( \mathcal{F}U, \mathcal{F}V \right) \simeq \mathbb{Z}^{\left( \text{hom}_{\textbf{Graphs}}(U,V) \right)},
\]
where $\mathbb{Z}^{(I)}$ denotes the direct sum of $\mathbb{Z}$ with coordinates from $I$ (i.e., the free group of rank $|I|$), and $\mathbb{Z}^{(\emptyset)}$ is understood to be the trivial group 0.  
\end{theorem}

Now we prove the  \eqref{item_TorGenSet} $\implies$ \eqref{item_VP} direction of Theorem \ref{thm_Main}.  Suppose Vop\v{e}nka's Principle fails; then by \cite{MR1294136} there is a rigid sequence
\[
\langle X_\alpha \ : \ \alpha \in \text{ORD} \rangle
\]
of graphs, rigid in the sense that $\text{hom}_{\textbf{Graphs}}(X_\alpha,X_\beta) = \emptyset$ whenever $\alpha \ne \beta$.  Let $\mathcal{F}$ be the functor from Theorem \ref{thm_PrzFunctor}.  Let $G_\alpha:= \mathcal{F}X_\alpha$ for each $\alpha \in \text{ORD}$.  Then
\begin{equation*}
\alpha \ne \beta \ \implies \ \text{hom}_{\textbf{Ab}}(G_\alpha,G_\beta)\simeq \mathbb{Z}^{\left(\text{hom}_{\text{graphs}} \left(X_\alpha,X_\beta \right) \right)}=\mathbb{Z}^{\left( \emptyset \right)} = 0.
\end{equation*}

Hence,
\[
\mathcal{G}:=\{ G_\alpha \ : \ \alpha \in \text{ORD} \}
\]
is a rigid proper class of abelian groups.  We prove that $T(\mathcal{G})$ (the torsion closure of $\mathcal{G}$) is not generated by a set.  Suppose, toward a contradiction, that there is a set $\mathcal{S} \subset T(\mathcal{G})$ such that $T(\mathcal{G}) = T(\mathcal{S})$.  By Lemma \ref{lem_HierarchLemma}, $T(\mathcal{G}) = T_\infty^{\mathcal{G}}$, so there is an ordinal $\gamma^*$ such that $\mathcal{S} \subseteq T^{\mathcal{G}}_{\gamma^*}$.  Then
\[
\mathcal{S} \subseteq T^{\mathcal{G}}_{\gamma^*} \subseteq T(\mathcal{G}) \text{ and } T(\mathcal{S}) = T(\mathcal{G}).
\]
\noindent It follows that 

\begin{equation}\label{eq_GenBySetTG}
T\left( T^{\mathcal{G}}_{\gamma^*} \right)=T(\mathcal{G}) .
\end{equation}

Since $\mathcal{G}$ is a rigid proper class and $T^{\mathcal{G}}_{\gamma^*}$ is a set, we may fix, for the remainder of the proof, some nonzero $G \in \mathcal{G} \setminus T^{\mathcal{G}}_{\gamma^*}$.  

\begin{globalClaim}
$T^{\mathcal{G}}_{\gamma^*} \subseteq {}^{\perp_0} G$.
\end{globalClaim}
\begin{proof}
We show by induction on $\alpha \le \gamma^*$ that $T^{\mathcal{G}}_\alpha \subseteq {}^{\perp_0} G$.  The only nontrivial step is the successor step.  Assume $T^{\mathcal{G}}_\alpha \subseteq {}^{\perp_0} G$ and that $\alpha +1 \le  \gamma^*$.  Then

\[
T^{\mathcal{G}}_{\alpha+1} \text{ is a set of representatives for } I\left( E \left( D_{\alpha} \left( T^{\mathcal{G}}_{\alpha} \cup \left( V_{\alpha+1} \cap \mathcal{G} \right) \right) \right) \right).
\]

Since $I(E(D_{\alpha}(-))$ preserves membership in ${}^{\perp_0} G$ by Fact \ref{fact_OperationsPreserve}, and the induction hypothesis tells us $T^{\mathcal{G}}_\alpha \subset {}^{\perp_0} G$, it suffices to show that $V_{\alpha+1} \cap \mathcal{G} \subset {}^{\perp_0} G$; and by rigidity of $\mathcal{G}$ it in turn suffices to know that $G \notin V_{\alpha+1} \cap \mathcal{G}$.  This is true, because $G \notin T^{\mathcal{G}}_{\gamma^*}$ and, by Lemma \ref{lem_HierarchLemma} and the assumption that $\alpha + 1 \le \gamma^*$, 
\[
V_{\alpha+1}  \cap \mathcal{G} \subseteq  V_{\gamma^*} \cap \mathcal{G} \subseteq T^{\mathcal{G}}_{\gamma^*}
\]
\end{proof}

Since ${}^{\perp_0} G$ is a torsion class and $T^{\mathcal{G}}_{\gamma^*} \subseteq {}^{\perp_0} G$ by the claim, minimality of torsion closure ensures  $T(T^{\mathcal{G}}_{\gamma^*}) \subseteq {}^{\perp_0} G$.  Together with \eqref{eq_GenBySetTG} this yields
\[
0 \ne G \in \mathcal{G} \subseteq T(\mathcal{G}) = T(T^{\mathcal{G}}_{\gamma^*}) \subseteq {}^{\perp_0} G,
\]
which is a contradiction, since $\text{id}_G$ is a nonzero member of $\text{hom}_{\textbf{Ab}}(G,G)$.

\section{Open Questions}

Given a class $\mathcal{X}$ of modules over a fixed ring, a module $G$ is called $\mathcal{X}$-Gorenstein Projective ($\mathcal{X}$-$\mathcal{GP}$) if there is an acyclic complex $P_\bullet$ of projective modules such that $G = \text{ker}(P_0 \to P_1)$ and $G \in {}^{\perp_1} \mathcal{X}$; the latter means that $\text{Ext}^1(G,X)=0$ for all $X \in \mathcal{X}$.  The Gorenstein Projective modules are the class $\mathcal{X}$-$\mathcal{GP}$, where $\mathcal{X}$ is the class of projective modules.  We refer to Cort{\'e}s-Izurdiaga and \v{S}aroch~\cite{cortes2023module} for more background on these classes, and to Cox~\cite{Cox_MaxDecon} for the proof that they have the crucial property of being eventually almost everywhere closed under quotients (despite failing, in general, to be closed under quotients).

The author's original motivation for introducing the Maximum Deconstructibility principle in \cite{Cox_MaxDecon} was to obtain the consistency, relative to large cardinals, of the scheme
\begin{quote}
 ``For any class $\mathcal{X}$ of modules (over any fixed ring), the class $\mathcal{X}$-$\mathcal{GP}$ is deconstructible".  
 \end{quote}

Does this scheme imply Vop\v{e}nka's Principle?  An affirmative answer would necessarily involve non-hereditary rings, since over a hereditary ring, for any class $\mathcal{X}$, $\mathcal{X}$-$\mathcal{GP}$ is equal to the class of projective modules, which is always deconstructible.  An affirmative answer would also need to involve some class $\mathcal{X}$ other than the class of projectives, since by \cite{cortes2023module}, a proper class of strongly compact cardinals (which is much weaker than Vop\v{e}nka's Principle) suffices to obtain deconstructibility of $\mathcal{GP}$ ($= \{ \text{projectives} \}$-$\mathcal{GP}$) over every ring.

\vspace{10em}

\noindent \textbf{Declarations}

Consent to Publish declaration: not applicable

Data availability: all cited references are publicly available

Author contribution:  all results due to the author

Competing Interests: none

\end{document}